\documentclass[11pt]{amsart}
\usepackage{enumitem}
\usepackage{amsmath,amsthm,amssymb,mathrsfs,amsfonts,verbatim,enumitem,color,leftidx,mathtools}
\usepackage{mathabx}
\usepackage{etoolbox} 
\usepackage{bbm}
\usepackage[all,tips]{xy}
\usepackage{graphicx,ifpdf}
\usepackage{stmaryrd}
\usepackage{amssymb}
\usepackage{dsfont}
\usepackage{float}

\ifpdf
\fi
\usepackage[colorlinks]{hyperref}
\hypersetup{
linkcolor=blue,          
citecolor=green,        
}
\usepackage{cleveref}

\usepackage{tikz}
\usetikzlibrary{calc}
\usetikzlibrary{shapes.misc}

\tikzset{cross/.style={cross out, draw=black, minimum size=2*(#1-\pgflinewidth), inner sep=0pt, outer sep=0pt},
cross/.default={1pt}}

\newtheorem{thm}{Theorem}[section]
\newtheorem{lem}[thm]{Lemma}
\newtheorem{cor}[thm]{Corollary}
\newtheorem{prop}[thm]{Proposition}

\newtheorem{conj}[thm]{Conjecture}

\theoremstyle{definition}
\newtheorem{defi}[thm]{Definition}

\newtheorem{remark}[thm]{Remark}

\theoremstyle{remark}

\numberwithin{equation}{section}

\definecolor{esperance}{rgb}{0.0,0.5,0.0}

\newcommand{\bw}{\mathbf{w}}

\newcommand{\N}{\mathbb{N}}
\newcommand{\Z}{\mathbb{Z}}

\newcommand{\sm}{\smallsetminus}

\newcommand{\al}{\alpha}

\newcommand{\Ga}{\Gamma}
\newcommand{\del}{\delta}

\newcommand{\Lam}{\Lambda}
\newcommand{\eps}{\epsilon}

\newcommand{\Om}{\Omega}
\newcommand{\vphi}{\varphi}

\newcommand{\bP}{\mathbb{P}}
\newcommand{\bR}{\mathbb{R}}
\newcommand{\bZ}{\mathbb{Z}}

\newcommand{\bN}{\mathbb{N}}

\newcommand{\SL}{\operatorname{SL}}
\newcommand{\SO}{\operatorname{SO}}

\newcommand\set[1]{\left\{#1\right\}}
\newcommand\pa[1]{\left(#1\right)}

\newcommand\av[1]{|#1|}
\newcommand\on[1]{\operatorname{#1}}

\newcommand\br[1]{\left[#1\right]}

\newcommand{\wstar}{\overset{\on{w}^*}{\lra}}

\newcommand{\defn}{\overset{\on{def}}{=}}

\newcommand{\lra}{\longrightarrow}

\newcommand{\onto}{\xymatrix{\ar@{>>}[r]&}}

\newcommand{\cfe}{c.f.e }

\newcommand{\cavg}[1]{\del_{#1}^{[0,2\log q]}}

\newcommand{\eqlabel}[2]
{
\begin{equation}
{#2}\label{#1}
\end{equation}
}

\newcommand\numberthis{\addtocounter{equation}{1}\tag{\theequation}}

\newcommand\mathitem{\item\leavevmode\vspace*{-0.75\dimexpr\baselineskip+\abovedisplayskip\relax}}

\begin{document}

\title[Continued fractions statistics of random rationals]{On the rate of 
convergence of continued fraction statistics of random rationals}

\date{}

\author{Ofir David}
\address{Department of Mathematics, Technion, Haifa, Israel {\it eofirdavid@gmail.com }}
\author{Taehyeong Kim}
\address{The Einstein Institute of Mathematics, Edmond J. Safra Campus, Givat Ram, The Hebrew University of Jerusalem, Jerusalem, 91904, Israel 
{\it taehyeong.kim@mail.huji.ac.il}}
\author{Ron Mor}
\address{The Einstein Institute of Mathematics, Edmond J. Safra Campus, Givat Ram, The Hebrew University of Jerusalem, Jerusalem, 91904, Israel {\it ron.mor@mail.huji.ac.il}}
\author{Uri Shapira}
\address{Department of Mathematics, Technion, Haifa, Israel 
{\it ushapira@technion.ac.il}}

\thanks{}


\def\thefootnote{}
\def\thefootnote{\arabic{footnote}}
\setcounter{footnote}{0}

\begin{abstract}
We show that the statistics of the continued fraction expansion of a randomly chosen rational in the unit interval, with a fixed large denominator $q$, approaches the Gauss-Kuzmin statistics with polynomial rate in $q$. 
This improves on previous results giving the convergence without rate. 
As an application of this effective rate of convergence, we show that the statistics of a randomly chosen rational in the unit interval, with a fixed large denominator $q$ and \textit{prime} numerator, also approaches the Gauss-Kuzmin statistics.
Our results are obtained as applications of improved non-escape of mass and equidistribution statements for the geodesic flow on the space $\SL_2(\bR)/\SL_2(\bZ)$.
\end{abstract}

\maketitle
\section{Introduction}
In this article we wish to revisit some of the results proved in \cite{DS} and strengthen them. The main novel observation in the current article is that one of the two assumptions appearing in the main theorems of \cite{DS} follows from the other and hence can be dropped. We then add to this by presenting a new arithmetic application of the stronger result.

\subsection{Background and the main application}
It is well known since ancient times that a real number $x\in(0,1)$ admits a \textit{continued fraction expansion} (abbreviated hereafter as c.f.e); namely, there are positive integers 
$a_1,a_2,\dots$ such that 
$$x = \frac{1}{
a_1 +\frac{1}{
a_2+\frac{1}{\ddots}
}
}.$$
In fact, to be more precise, irrational $x$'s correspond bijectively to infinite sequences $(a_n)\in \bN^\bN$, where the above equation is interpreted as a limit, and rational $x$'s correspond to finite sequences of positive integers, where here the correspondence is not 
$1$-$1$ but almost so: each rational 
$x\in (0,1)$ can be written in exactly two ways, 
$$x = \frac{1}{
a_1 +\frac{1}{
a_2+\frac{1}{\ddots+\frac{1}{a_n}} 
}
} \textrm{ or } 
x = \frac{1}{
a_1 +\frac{1}{
a_2+\frac{1}{\ddots+\frac{1}{(a_n-1) + \frac{1}{1}}}
}
}
$$
where $n\in \bN$ and $a_n\ge2$.
As usual, we abbreviate and write $x = \br{a_1,a_2,\dots}$ when $x$ is irrational and 
$x = \br{a_1,\dots, a_n}$ when $x$ is rational. As a convention we will always use the shorter expansion for rationals, which does not end with $1$.
The length of the \cfe of a rational number is defined by 
\begin{equation}\label{eq: length of cfe}
\on{len}([a_1,\ldots,a_n])\defn n.
\end{equation}

Our discussion will revolve around the continued fraction expansions of rationals with large denominators, but to motivate the phenomena discussed it is actually more natural to start with irrationals. 

It is well known (see for example the introduction of \cite{AS}) that given any finite string 
of positive integers $\bw = (w_1,\dots, w_k)\in \bN^k$, for Lebesgue almost any irrational $x\in (0,1)$, the
asymptotic frequency of appearance of $\bw$ in the \cfe of $x$ exists and is independent of $x$. That is, the limit 
\begin{equation}\label{eq: asymptotic density of w}
D_\bw \defn \lim_{n\to\infty} \frac{\#\set{1\le \ell \le n: \bw = (a_\ell,a_{\ell+1},\dots, a_{\ell+k-1})}}{n}
\end{equation}
exists and is independent of $x$, as the notation suggests (in fact $D_{\bw}$ equals $\frac{1}{\log 2}\int_{I_\bw} \frac{1}{1+t}dt$ where $I_\bw$ is a certain interval in $(0,1)$, see Remark~\ref{remark:gauss_measure}). For example, for a randomly chosen $x$, the asymptotic density of appearance of $1$'s in its \cfe is $2-\frac{\log3}{\log2} \sim 0.415$ with probability $1$.

A similar density could be defined for rationals: if $x = \br{a_1,\dots, a_n}$ then we can define
$$D_\bw(x) \defn \frac{\#\set{1\le \ell \le n-k+1: \bw = (a_\ell,a_{\ell+1},\dots, a_{\ell+k-1})}}{n-k+1}.$$
In \cite{DS} the authors investigated the densities $D_\bw(j/q)$ of rationals $j/q$ in reduced form for large denominators $q$, as well as the length of the \cfe of such rationals.
In particular, the following result is proved.
\begin{thm}[\cite{DS}]\label{thm 1}
For any $q\in\bN$, let $j/q\in (0,1)$ be a random rational with denominator $q$ and numerator $j$ co-prime to $q$ chosen uniformly (namely, one picks $j$ randomly and uniformly from the $\vphi(q)$ possible numerators co-prime to $q$). Then for any $k\in \bN$, any string $\bw\in \bN^k$, and any $\eps>0$, we have 
\begin{enumerate}
    \mathitem \[\bP(\av{D_\bw(j/q) - D_\bw}>\eps)\to 0 \textrm{ as }q\to \infty,\]
    \mathitem \[\bP\pa{ \av{\frac{\on{len}(j/q)}{\log q} - \frac{12\log 2}{\pi^2}} >\eps}\to 0 \textrm{ as }q\to \infty.\]
\end{enumerate}
\end{thm}

In this paper, we are able to improve this result, and show that not only the probability in Theorem~\ref{thm 1} converges to zero, but it does so at a polynomial rate.
\begin{thm}\label{thm 1 exp}
    For any $q\in \bN$, let $j/q\in(0,1)$
    be a random rational with denominator $q$ and numerator $j$ co-prime to $q$ chosen uniformly. Then, for any
    $k\in \bN$, any string $\bw\in \bN^k$, and any $\eps>0$, there exist $\alpha_{\epsilon}>0$ and $\alpha_{\epsilon,\bw}>0$ such that  for all $q$ large enough:
    \begin{enumerate}
        \mathitem \[\bP(\av{D_\bw(j/q) - D_\bw}>\eps)<q^{-\alpha_{\epsilon,\bw}},\]
        \mathitem\label{item:1.2.2} \[\bP\pa{ \av{\frac{\on{len}(j/q)}{\log q} - \frac{12\log 2}{\pi^2}} >\eps}<q^{-\alpha_{\epsilon}}.\]
    \end{enumerate}
\end{thm}

    Our methods do not provide any estimate of the values $\alpha_{\eps}$ and $\alpha_{\eps,\bw}$ in Theorem~\ref{thm 1 exp}, only their existence. It would be interesting to find out how fast do these probabilities must decay, and in particular the dependence of these powers on $\eps$.

Theorem~\ref{thm 1 exp} is in fact equivalent to the following generalization of Theorem~\ref{thm 1}. 
We show that the probability for the \cfe of $j/q$ to `behave badly' (in the sense of Theorem~\ref{thm 1}) approaches $0$, even if we choose $j$ uniformly at random from a set $\Lam_{q}\subseteq (\Z/q\Z)^{\times}$ of smaller size $q^{-\alpha_{q}}\varphi(q)$ with $\alpha_{q}\to 0$, instead of choosing from $(\Z/q\Z)^{\times}$.
Here and throughout this paper, we identify between $(\Z/q\Z)^{\times}$ and the set $\set{1\le j\le q:\ \gcd(j,q)=1}$.
\begin{thm}\label{thm 3}
    For any $q\in \bN$, let $\Lam_q\subseteq (\bZ/q\bZ)^{\times}$ so that $\underset{q\to\infty}{\lim} \frac{\log \av{\Lam_q}}{\log q} = 1$, and let $j/q$
    be a random rational with denominator $q$ and numerator $j\in \Lam_q$ chosen uniformly (we denote this probability measure over $\Lam_{q}$ by $\bP|_{\Lam_q}$). Then, for any
    $k\in \bN$, any string $\bw\in \bN^k$, and any $\eps>0$, we have
    \begin{enumerate}
        \mathitem \[\bP|_{\Lam_q}(\av{D_\bw(j/q) - D_\bw}>\eps)\to 0 \textrm{ as }q\to \infty,\]
        \mathitem \[\bP|_{\Lam_q}\pa{ \av{\frac{\on{len}(j/q)}{\log q} - \frac{12\log 2}{\pi^2}} >\eps}\to 0 \textrm{ as }q\to \infty.\]
    \end{enumerate}
\end{thm}

As an application, we deduce the following \textit{prime numerator} version of Theorem~\ref{thm 1}. It follows from the prime number theorem in combination with Theorem~\ref{thm 3}. 
\begin{thm}\label{thm 2}
For any $q\in\bN$, let $p/q\in (0,1)$
    be a random rational with denominator $q$ and prime numerator $p$ co-prime to $q$ chosen uniformly. Then for any $k\in \bN$, any string $\bw\in \bN^k$, and any $\eps>0$, we have
    \begin{enumerate}
        \mathitem \[\bP|_{\text{primes}}(\av{D_\bw(p/q) - D_\bw}>\eps)\to 0 \textrm{ as }q\to \infty,\]
        \mathitem \[\bP|_{\text{primes}}\pa{ \av{\frac{\on{len}(p/q)}{\log q} - \frac{12\log 2}{\pi^2}} >\eps}\to 0 \textrm{ as }q\to \infty.\]
    \end{enumerate}
\end{thm}

    Related questions regarding the distribution of the coefficients in continued fractions of random rationals, and particularly the expected length of the expansion, were previously discussed in the literature. 
    In~\cite{Heilbronn1969length}, Heilbronn computed the asymptotics of the average length over $p\in(\Z/q\Z)^{\times}$, also giving an error term. This error term was improved by Porter~\cite{Porter1975Heilbronn} and then further improved by Ustinov~\cite{Ustinov2008Solutions}.

    Other works considered the length question on different sets. In~\cite{bykovski2005dispersion}, Bykovski considered for any $q$ the average length for the \cfe of $\frac{p}{q}$, over all $p$ not necessarily coprime to $q$.    
    Others considered a larger probability space, comprised of all pairs $(p,q)$ where $p$ is coprime to $q$ and $q\leq Q$ for some $Q$. 
    On this space, Dixon~\cite{dixon1970number} showed a large deviation result, namely bounded the probability that $\on{len}(p/q)$ deviates from the expected value by a certain amount. This result was later improved by Hensely~\cite{hensley1994number}. Furthermore, in~\cite{baladi2005} Baladi and Vall{\'e}e consider a central limit theorem on this space.
    In comparison with these works, we consider too a large deviation question, but on a significantly smaller probability space where the denominator is considered fixed. Furthermore, we also consider the distribution of words in the \cfe expansions rather than only its length. 

    For construction of normal numbers with respect to the \cfe using rational numbers, see~\cite{adler1981construction,vandehey2016construction}.

\subsection{The diagonal flow and the main result}
The proof of Theorem~\ref{thm 3} will be given in \S\ref{sec: proof of theorem 3} and is based on the dynamical results we now present.

We start with the setting in which our mathematical discussion will take place, namely
the space of unimodular lattices in Euclidean plane. 
We will use similar notation to \cite{DS}. 
Let $G=\SL_2(\bR)$, $\Ga=\SL_2(\bZ)$, $X=G/\Ga$, and let $x_0\in X$ denote the identity coset.
For a point $x\in X$ we let $\del_x$ denote the probability measure supported on the singleton $\set{x}$.
Consider the diagonal flow
$$
a_t = \begin{pmatrix} e^{t/2} & 0 \\ 0 & e^{-t/2} \end{pmatrix} \in G.
$$
The unstable horospherical subgroup $U$ of $G$ with respect to $a_t$ is given by
\[
U=\left\{u_s=\begin{pmatrix} 1 & s \\ 0 & 1 \end{pmatrix} : s\in \bR\right\}.
\]
For a given set $\Lam_q \subseteq (\bZ/q\bZ)^{\times}$ denote 
\begin{equation}\label{eq: discrete avg}
\del_{\Lam_q}^{[0, 2\log q ]} = \frac{1}{\av{\Lam_q}}\sum_{p\in \Lam_q} \frac{1}{2\log q } \int_0^{2\log q} \del_{a_{t}u_{p/q}x_0}dt,
\end{equation}
that is, $\del_{\Lam_q}^{[0, 2\log q ]}$ is the probability measure given by the continuous average in the diagonal direction on the segment $[0, 2\log q]$ of the discrete counting probability measure on the points $\{u_{p/q}x_0 : p\in\Lambda_q\}$.  
For an illustration of the geodesic segments on which we average, see Figure~\ref{fig: geodesic_average} and the discussion before it. 
Recall that we say that a sequence of probability measures $(\mu_n)_{n\in\N}$ on $X$ \textit{does not exhibit escape of mass} if any weak* 
limit of it is a probability measure.
The following is one of the main results of \cite{DS}.
\begin{thm}\cite[Theorem 1.7]{DS}\label{Thm_DS}
Let $\Lam_q \subseteq (\bZ/q\bZ)^{\times}$ be subsets such that
\begin{enumerate}
    \item $\underset{q\to\infty}{\lim} \frac{\log |\Lam_q|}{\log q}=1$,
    \item the sequence of measures $\del_{\Lam_q}^{[0,2\log q]}$ does not exhibit escape of mass.
\end{enumerate}
Then $\del_{\Lam_q}^{[0,2\log q]} \wstar \mu_X$, where $\mu_X$ is the Haar probability measure on $X$.
\end{thm}

The current paper concerns itself with proving that the second item in Theorem~\ref{Thm_DS} is superfluous since it actually follows from the first item. Our main result is the following bound on the possible escape of mass for divergent orbits.
\begin{thm}\label{Thm_nonescape}
Let $h\in[0,1]$. Suppose that $\Lam_q\subseteq (\bZ/q\bZ)^\times$ are given so that  
\[
\liminf_{q\to\infty} \frac{\log |\Lam_q|}{\log q} \geq h.
\]
Then any weak$^\ast$ limit $\mu$ of the sequence $\del_{\Lam_q}^{[0, 2\log q ]}$ satisfies
\[
\mu(X) \geq 2h-1.
\]
\end{thm}
Combining Theorem \ref{Thm_DS} and Theorem \ref{Thm_nonescape} with $h=1$, we have
\begin{cor}\label{cor: main}
Let $\Lam_q\subseteq (\bZ/q\bZ)^\times$ be subsets satisfying 
$\lim_{q\to\infty} \frac{\log |\Lam_q|}{\log q}=1$, then $\del_{\Lam_q}^{[0,2\log q]} \wstar \mu_X$.
\end{cor}
\begin{remark}
    We note here that the assumption $\lim_{q\to \infty}\frac{\log|\Lam_q|}{\log q}=1$ in Theorem~\ref{Thm_DS} and Corollary~\ref{cor: main} is necessary and refer the reader to \S\ref{sec:examples} for a discussion on the matter.
\end{remark}

To illustrate the trajectories $\{a_{t}u_{p/q}x_0:\ t\in [0,2\log q]\}$, we show in Figure~\ref{fig: geodesic_average} their projection to the upper half plane $\mathbb{H}$, which is identified with the quotient $K\backslash G$, where $K=\SO_{2}(\mathbb{R})$. The trajectories are shown in the standard fundamental domain, given by points $z\in\mathbb{H}$ with $|\Re(z)|\leq \frac{1}{2}$ and $|z|\geq 1$. 
As discussed in~\cite{DS}, at times $t=0,2\log q$, these trajectories start and end at height $y=1$ in the upper half plane, and for $t\not\in[0,2\log q]$ the trajectories diverge directly into the cusp. Hence the $[0,2\log q]$-portion of the geodesic trajectory is the interesting one to study.

As can be observed, there is a variance between the trajectories for different values of $p$, with some comprising of just one long excursion to the cusp, and others disproportionally bounded. This is why averaging on a relatively large set is required to obtain equidistribution, and otherwise one can construct counter examples as we have in \S\ref{sec:examples} . 
Furthermore, there are visible symmetries between the trajectories for different values of $p$, which were proved in~\cite{DS} and are also used in this current paper (see Lemma~\ref{lem: half symmetry}).

\begin{figure}[H]
\includegraphics[height=4cm]{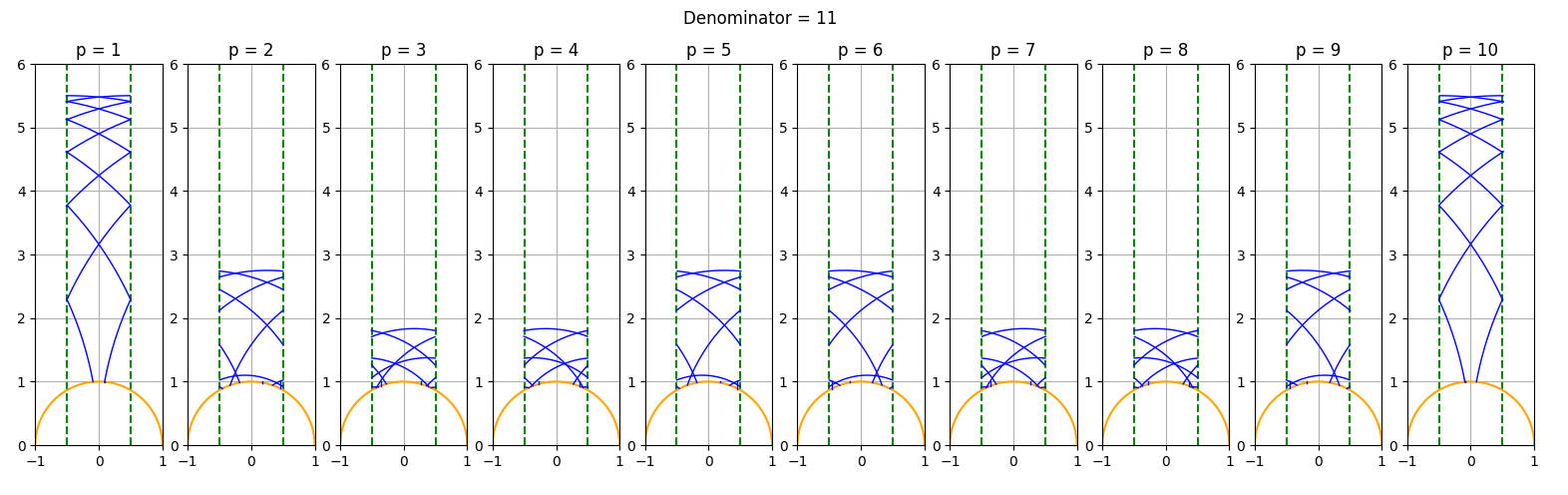}
\centering
\caption{The projection to the upper half plane, of the trajectories $\{a_{t}u_{p/q}x_0:\ t\in[0,2\log q]\}$, for $q=11$ and all $p\in(\Z/q\Z)^{\times}$.}
\label{fig: geodesic_average}
\end{figure}

\subsection{Prime numerator Zaremba}\label{sec: Zaremba}
In this part of the introduction we wish to present a conjecture we refer to as the \textit{prime numerator Zaremba} conjecture. The discussion in this subsection has a different flavour to it as we do not present any theorems but rather discuss a conjecture and present supporting data. Nevertheless, we find the conjecture interesting and think it fits in this paper as we were inspired by Theorem~\ref{thm 2} when
we formulated it. We hope it will inspire other people to look into it as well. 

The Zaremba conjecture asserts that for any denominator $q$, one can find $j\in (\bZ/q\bZ)^\times$ such that the \cfe of $j/q=\br{a_1,\dots, a_n}$ satisfies $a_i\le 5$ for 
all $1\le i\le n$. Lets call such a $j$, a $5$-Zaremba numerator for the denominator $q$. 
This conjecture is in direct contrast to Theorem~\ref{thm 1} which implies 
in particular, that the set of $5$-Zaremba numerators in $(\bZ/q\bZ)^\times$ has size bounded by $o(q^{1-\epsilon})$ as $q\to\infty$ and in particular its proportion in $(\bZ/q\bZ)^{\times}$ approaches zero.  Intrigued by this, we used the computer to check if one should expect the existence of $5$-Zaremba \textit{prime} numerators and it seems that the data suggests so. We therefore
suggest:
\begin{conj}
\label{conj: prime Zaremba}
For any $q>1$, there exists a prime residue $p\in (\bZ/q\bZ)^\times$ which is $5$-Zaremba for $q$. Moreover, for any $q>1$ which is not one of $\{6, 54, 150\}$ there exists a 4-Zaremba prime numerator.
\end{conj}

\begin{figure}
\includegraphics[width=6cm, height=5cm]{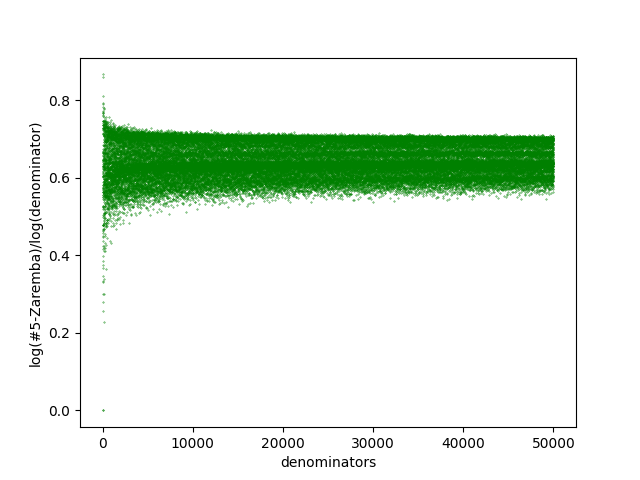}
\includegraphics[width=6cm, height=5cm]{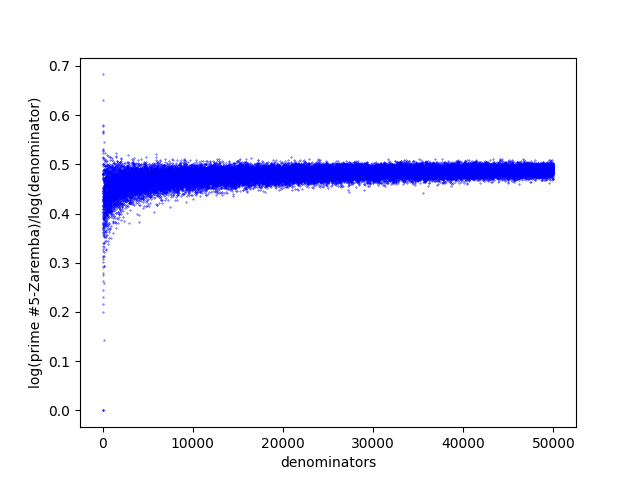}
\centering
\caption{Counting the 5-Zaremba numerators and plotting the $\frac{\log(\#numerators)}{\log(denominator)}$. On the left (green) all numerators, and on the right (blue) prime numerators.}
\label{fig: zaremba}
\end{figure}

The data, as can be seen in Figure~\ref{fig: zaremba}, suggests that not only there are Zaremba numerators for any $q$, but in general the number of them increases as the denominator grows. 
A related statement in this regard is to study the asymptotic number of Zaremba numerators for all $1\leq q\leq Q$, and show it grows fast enough as a function of $Q$ in order for the Zaremba conjecture to hold. Indeed, without the prime restriction on the numerators, this  was found  by Hensley to be true~\cite{HensleyBoundedDigits}. It would be interesting to know if it still holds with the prime numerator restriction.

One may speculate from the data in Figure~\ref{fig: zaremba} that the graphs converge to certain limits $\delta_{5}$ and $\delta_{5,P}$, respectively, as $q\to\infty$. 
If the limit $\delta_{5}$ exists, it easily follows from Hensley's work that the relation $\delta_{5}=2\dim_{H}E_{5}-1$ must be satisfied, where $E_{5}$ is the set of all irrationals in $(0,1)$ for which all digits in their \cfe are bounded by $5$.
It is therefore interesting to wonder if $\delta_{5,P}$ is also related to the Hausdorff dimension of some set. If so, it seems reasonable to guess from our data that this dimension should be strictly smaller than $\dim_{H} E_{5}$. Similar question can be asked for bounds other than~$5$.

\subsection{Counter examples}\label{sec:examples}
In this section we give examples for the necessity of the assumption $\lim_{q\to\infty}\frac{\log|\Lam_q|}{\log q}=1$ in Theorem~\ref{Thm_DS} and Corollary~\ref{cor: main}. In~ \S\ref{sec:Escape of mass} we show that without this assumption escape of mass can occur, and in~\S\ref{sec:equidistribution fail} that even if there is no escape of mass, equidistribution can still fail.

\subsubsection{Escape of mass}\label{sec:Escape of mass}
    Fix any $0<\alpha<1$. We consider \[\Lam_{q}=\{p\in (\Z/q\Z)^{\times}:\ p<q^{\alpha}\}\]
    and show that the measures $\delta_{\Lam_{q}}^{[0,2\log q]}$ exhibit escape of mass. 
    The size of these sets satisfies
    \[\lim_{q\to\infty}\frac{\log|\Lam_{q}|}{\log q}=\alpha,\]
    which is clear for prime $q$, and for general $q$ follows by using some inclusion-exclusion argument (as shown for example in Lemma \cite[Lemma 2.12]{DS}).

    For escape of mass, recall Mahler's compactness theorem (c.f.~\S\ref{sec: escape}) which says that an element $gx_0 \in X$ is near the cusp if the corresponding lattice $g\mathbb{Z}^2$ has a nontrivial short vector. More precisely, for any compact subset $Y\subset X$, there is some $M>1$ so that any $gx_0\in Y$ satisfies $\underset{0\neq v\in \Z^{2}}{\min}\|gv\|_{\infty}\geq \frac{1}{M}$.

    Note that any $p\in \Lam_{q}$ and $t\in \big(2\log M, 2\log q^{(1-\alpha)}-2\log M\big)$ satisfy
    \[\|a_{t}u_{p/q}\begin{pmatrix}
        0\\ 1
    \end{pmatrix}\|=\|\begin{pmatrix} e^{t/2}\frac{p}{q} \\ e^{-t/2} \end{pmatrix}\|<\frac{1}{M},\]
    i.e.\ the $[0,2\log q]$-trajectory of $u_{p/q}x_0$ starts with a long excursion to cusp, outside of $Y$, for approximately a $1-\alpha$ fraction of the total time.
    It follows that for any compact $Y\subset X$
    \[\limsup_{q\to\infty}\delta_{\Lam_{q}}^{[0,2\log q]}(Y)\leq \alpha,\]
    and therefore $\mu(X)\leq \alpha$ for any weak* limit $\mu$ of $\delta_{\Lam_{q}}^{[0,2\log q]}$, showing escape of mass and in particular preventing equidistribution.

\subsubsection{Equidistribution}\label{sec:equidistribution fail}
We show next that even if there is no escape of mass, equidistribution still may fail.
    For a denominator $q$, consider the $n$-Zaremba numerators as in ~\S\ref{sec: Zaremba}: 
    \[\Lambda_{q,n}=\{p\in (\Z/q\Z)^{\times}:\ \text{all digits in the \cfe of $p/q$ are at most $n$}\}.\]
    Let \[R_{n}(Q)\coloneqq \{(p,q):\ p\in \Lambda_{q,n},\ 1\leq q\leq Q\}.\]
    By a result of Hensely~\cite{Hensley1990II}, we have
    \[\lim_{Q\to \infty}\frac{|R_{n}(Q)|}{Q^{2 \dim_{H}E_n}}=c_n\]
    for some constant $c_n>0$, where $E_n$ is the set of all irrationals in $(0,1)$ for which all digits in their \cfe are bounded by $n$.
    As \[|R_{n}(Q)|=\sum_{1\leq q\leq Q}|\Lambda_{q,n}|\leq Q\cdot \max_{1\leq q\leq Q}|\Lambda_{q,n}|,\]
    it follows from Hensley's result that there exists an increasing sequence $(q_i)_{i=1}^{\infty}$ so that 
    \[\liminf_{i\to\infty}\frac{\log |\Lam_{q_i,n}|}{\log q_i}\geq 2\dim_{H}E_{n}-1.\]
    It is well known~\cite{Jarnik1928diophantine} that \[\lim_{n\to\infty}\dim_{H} E_{n}=1.\]
    Therefore, by choosing $n$ large enough, the collection $\{\Lambda_{q_i,n}\}_{i=1}^{\infty}$ can be assumed to be of logarithmic size as close to $1$ as we would like.

    However, due to the small digits in the \cfe of $\frac{p}{q_i}$ for $p\in\Lam_{q_i,n}$ there is some compact set which depends only on $n$, on which $\delta_{\Lam_{q_i,n}}^{[0,2\log q_i]}$ is supported for all $i$ (c.f.\ the proof of Theorem~1.4 in~\cite[pg.\ 173]{DS}). It follows that any weak* limit of $\delta_{\Lam_{q_i},n}^{[0,2\log q_i]}$ is a compactly supported probability measure, in particular not equal to $\mu_{X}$. 

\subsection{Structure of the paper} In \S\ref{sec: escape} we prove Theorem~\ref{Thm_nonescape}, which as discussed above implies the strengthening of the equidistribution result as in Corollary~\ref{cor: main}. In \S\ref{sec: proof of theorem 3} we deduce 
Theorems~\ref{thm 1 exp}-\ref{thm 2} which are the applications for the \cfe of rationals.

\subsection*{Acknowledgments}
We are grateful to Elon Lindenstrauss for suggesting to us to look into improving Theorem~\ref{Thm_DS}. This work has received funding from the European Research Council (ERC) under the European Union’s Horizon 2020 Research and Innovation Program, Grant agreement no. 754475.
R.M.\ acknowledges the support by ERC 2020 grant HomDyn (grant no.~833423).

\section{Non-escape of mass for divergent orbits}\label{sec: escape}
In this section we prove Theorem~\ref{Thm_nonescape}. 
The key input required for the proof is the relation between entropy and escape of mass, which was studied in various previous works. The $G/\Gamma=\SL_{2}(\bR)/\SL_{2}(\bZ)$ case which we consider in this work, was studied in~\cite{ELMV2011distribution} by Einsiedler, Lindenstrauss, Michel and Venkatesh. Since then, generalizations were obtained for other rank $1$ situations~\cite{EKP2015escape,Mor2022excursions}, as well as for the high rank case~\cite{EK2012entropy,KKLM,Mor2023a}. We give here a brief overview of the ingredients we use and the idea of the proof.

For $N\in\bN$ and $\eta>0$, let 
$B_{N,\eta}^{+}=\bigcap_{n=0}^{N}a_{-n}B_{\eta}^{G}a_{n}$, where $B_{\eta}^{G}$ is the ball in $G$ of radius $\eta$ around the identity.
Define a forward Bowen $(N,\eta)$-ball to be a set of the form $B_{N,\eta}^{+}z$ for $z\in G/\Ga$.
For $\Lambda_{q}\subseteq (\bZ/q\bZ)^{\times}$, we consider the set $X_{\Lambda_{q}}\coloneqq\{u_{p/q}x_0:\ p\in \Lambda_{q}\}$.
It is straightforward to show~\cite[Lemma 2.10]{DS} that the points of $X_{(\bZ/q\bZ)^{\times}}$ are well separated, in the sense that two distinct points cannot belong to the same forward Bowen $(N_q,\eta)$-ball, where $N_{q}=\lfloor \log q \rfloor$, assuming $\eta$ is small enough.
Therefore, instead of studying the size of a given subset of $X_{(\bZ/q\bZ)^{\times}}$, we study the size of its covering by balls.

Note that a-priori, an amount of $\ll_{\eta} e^{N_q}$ forward Bowen $(N_{q},\eta)$-balls is required to cover a given compact subset of $X$, as each such ball has a small diameter of $\eta e^{-N_q}$ in the unstable direction. 
However, the key idea that was discovered in previous works is that if we only want to cover the points which spend a large proportion of their time high up in the cusp, we can cut down the number of balls. 
Quantitatively, if we only consider points which spend more than $\kappa N_{q}$ time in the cusp, the required number of balls is $\ll_{\eta} e^{(1-\kappa/2)N_q}$.
We will use this idea in the proof of Theorem~\ref{Thm_nonescape}, where we bound the escape of mass rate, by showing that most points of $X_{\Lambda_{q}}$ do not spend a lot of time in the cusp during the interval $[0,N_{q}]$.

\medskip

In our proof, we will use the version of the aforementioned covering result as given in~\cite{KKLM}.
For that, we first define the discrete average in the diagonal direction with $N$ steps of size $t$ of a measure $\nu$ to be 
\[\nu^{N,t}  = \frac{1}{N} \sum _{k=0}^{N-1} a_{tk}*\nu, \] 
and the continuous average on the segment $[T_1,T_2]$  
\[\nu^{[T_1,T_2]}  = \frac{1}{T_2-T_1} \int_{T_1}^{T_2} (a_{t}*\nu) dt.\] 

Next, consider the distance $d_{U}$ on $U$ induced by the absolute value on $\bR$, and let 
$B_r^{U}$ be the ball in $U$ of radius $r$ centered at the identity. Given a compact subset $Q\subset X$, $N\in \bN$, $\kappa \in (0,1)$, $t>0$, and $x\in X$, define the set of points which spend at least a fraction $\kappa$ of their discrete trajectory outside $Q$ by
\[
Z_x(Q,N,t,\kappa):=\left\{u\in B_1^{U} : \delta_{ux}^{N,t}(X\smallsetminus Q) \geq\kappa \right\},
\]
and similarly for $T>0$ the continuous version
\[C_{x}(Q,T,\kappa)\coloneqq\left\{u\in B_1^{U} : \delta_{ux}^{[0,T]}(X\smallsetminus Q) \geq\kappa \right\}.\]

Lastly, let $\al_1 :X\to \bR_{>0}$ be defined by $\al_1(x)=\max\{\frac{1}{\|v\|}:v\in x\smallsetminus\{0\}\}$, for $\|\cdot\|$ the Euclidean norm, where we consider $X$ as the space of unimodular lattices in $\bR^2$.

Then we  have the following covering counting theorem based on~\cite[Theorem 1.5]{KKLM}.
\begin{thm}\label{Thm_KKLM}
There exist $t_0>0$ such that the following holds. For any $t>t_0$ there exists a compact set $Q:=Q(t)$ of $X$ such that for any $N\in\bN$, $\kappa\in(0,1)$, and $x\in X$, the set $Z_x(Q,N,t,\kappa)$ can be covered with \[\al_{1}(x)e^{t/2}(t/2)^{3N}e^{(1-\frac{\kappa}{2})tN}\]
many balls in $U$ of radius $e^{-t(N-1)}$.
\end{thm}
\begin{remark}\label{remark:KKLM}
    The original statement of~\cite{KKLM} used slightly different notations, with $a_{t}$ being the geodesic flow with a doubled rate, and $Z_{x}(Q,N,t,\kappa)$ defined by the requirement $\delta_{a_{t}ux}^{N,t}(X\smallsetminus Q)\geq\kappa$. 
    However, the translation to the statement of Theorem~\ref{Thm_KKLM} is straightforward as we show next.
\end{remark}
\begin{proof}[Proof of Theorem~\ref{Thm_KKLM}]
Observe that
\[a_{-t}Z_x(Q,N,t,\kappa)a_t \subset \left\{u\in B_1^U : \del_{a_t u a_{-t}x}^{N,t}(X\smallsetminus Q)\geq \kappa\right\}. \]
By \cite[Theorem 1.5]{KKLM}, the right hand side set can be covered with $\al_1(a_{-t}x)(t/2)^{3N}e^{(1-\frac{\kappa}{2})tN}$ balls in $U$ of radius $e^{-tN}$. Since $\al_1(a_{-t}x) \leq e^{t/2}\al_1(x)$, the theorem follows by conjugating with $a_t$. 
\end{proof}

It is convenient to work with the the following standard notations for the cusp structure of $X$. 
For any $M>0$, denote 
$$X^{\leq M} = \{x\in X : \al_1(x)\leq M \} \quad\text{and}\quad 
X^{> M} = \{x\in X : \al_1(x)> M \}.$$ 

While Theorem~\ref{Thm_KKLM} considers discrete averages, we would like to apply it to continuous averages. For that we have the following lemma.
\begin{lem}\label{lem_small_step}
\begin{enumerate}
    \item \label{item:1}
Let $x\in X$, $\kappa\in(0,1)$, $M>0$, and  $T,t\in\bR$ with $T\geq t>0$. Then
    $$C_{x}(X^{\leq M},T,\kappa)\subseteq Z_x(X^{\leq Me^{-t}}, \lfloor \frac{T}{t} \rfloor, t, \kappa).$$
    \item \label{item:2}
    Assume in addition that $t>t_0$, for $t_0$ as in Theorem~\ref{Thm_KKLM}. Then there are $M_{0},T_0>0$ which depend on $t$, so that for any $M\geq M_0$ and $T\geq T_0$, the set $C_{x}(X^{\leq M},T,\kappa)$ can be covered by 
    \[\alpha_{1}(x)\exp\Big((1-\frac{\kappa}{2}+\frac{3\log t}{t})T\Big)\]
    many balls in $U$ of radius $e^{-(T-2t)}$.
    \end{enumerate}
\end{lem}
\begin{proof}
In order to prove item~\eqref{item:1}, we will show that
    \[\delta_{x}^{\lfloor T/t\rfloor,t}(X^{\leq Me^{-t}})\leq\delta_{x}^{[0,T]}(X^{\leq M}), \]
    as the desired result follows immediately from this claim and the definitions.
    
    In order to prove the claim above, observe that for any $x\in X$, 
    \[\al_1(a_t x) \leq e^{t/2}\al_1(x).\]
    It follows that if $y\in X^{\leq Me^{-t}}$, then $\{a_{s}y:\ s\in[0,2t]\}\subset X^{\leq M}$.
    Let \[I=\{0\leq n< \lfloor\frac{T}{t} \rfloor : a_{nt}x \in X^{\leq Me^{-t}}\},\] so that $\frac{|I|}{\lfloor T/t \rfloor}=\delta_{x}^{\lfloor T/t\rfloor,t}(X^{\leq Me^{-t}})$. If $I=\emptyset$, then $\delta_{x}^{\lfloor T/t\rfloor,t}(X^{\leq Me^{-t}})=0$ and the claim is trivial.
    Otherwise, letting $n_0 = \max(I)$ we get that 
    $$\left(\bigsqcup_{n_0\neq n\in I} \left[nt,(n+1)t\right)\right) \sqcup [n_0t, (n_0+1+\left\{\frac {T}{t}\right\})t) \subseteq \{ s\in [0,T] : a_sx\in X^{\leq M}  \},$$
    where $\{y\}:=y-\lfloor y\rfloor$  for any $y\in\bR$.
    It follows that
    \begin{align*}
    T\cdot \delta_x^{[0,T]}(X^{\leq M}) & \geq t\cdot( |I| + \left\{\frac{T}{t}\right\})=t\cdot( \left\lfloor \frac{T}{t}\right\rfloor \cdot \frac{|I|}{\lfloor T/t\rfloor} + \left\{\frac{T}{t}\right\}) \\ &\geq t\cdot( \left\lfloor \frac{T}{t}\right\rfloor  + \left\{\frac{T}{t}\right\})\cdot \frac{|I|}{\lfloor T/t\rfloor}=T\cdot \delta_{x}^{\lfloor T/t\rfloor,t}(X^{\leq Me^{-t}}).
    \end{align*}
    Dividing by $T$ proves the claim, and thus completes the proof of item~\ref{item:1}.
    
    \medskip  
    
    To prove item~\ref{item:2}, let $Q=Q(t)$ be as in Theorem~\ref{Thm_KKLM}.
    Let $M_0=M_0(t)>0$ large enough so that $Q(t) \subseteq X^{\leq M_0 e^{-t}}$. Then, for any $M\geq M_0$, we have by definition and item~\ref{item:1}
    \eqlabel{eq:2.3.item2}{C_{x}(X^{\leq M},T,\kappa)\subseteq Z_x(Q, \lfloor \frac{T}{t} \rfloor, t, \kappa).}
    By Theorem~\ref{Thm_KKLM}, the RHS of~\eqref{eq:2.3.item2}, and hence the LHS as well, can be covered by 
    \[\alpha_{1}(x)e^{t/2}(t/2)^{3\left\lfloor\frac{T}{t}\right\rfloor}\exp\Big( (1-\frac{\kappa}{2}) t\left\lfloor\frac{T}{t}\right\rfloor\Big) \leq  \alpha_{1}(x)\exp\Big((1-\frac{\kappa}{2}+\frac{3\log (t/2)}{t}+\frac{t}{2T})T\Big)\]
    many balls in $U$ of radius
    \[e^{-t(\lfloor \frac{T}{t}\rfloor-1)}\leq e^{-(T-2t)}.\]
    Assuming $T$ is large enough compared to $t$, the size of this covering is also bounded by
    \[\alpha_{1}(x)\exp\Big((1-\frac{\kappa}{2}+\frac{3\log t}{t})T\Big),\]
    as desired.
\end{proof}

The previous lemma handled the covering of all points in the entire open ball $B_{1}^{U}$ (with certain restrictions on the time spent in the cusp).
To prove the results of this paper we need to  consider instead discrete averages over rational points, as follows. Given a set $\Lam\subseteq (\bZ/q\bZ)^\times$, we define
\[
\del_\Lam\defn \frac{1}{\av{\Lam}}\sum_{p\in \Lam} \del_{u_{p/q}x_0},
\]
with $q$ implicit in the notation (note that this definition agrees with~\eqref{eq: discrete avg}). Using that previous covering result, we show that the measures $\delta_{\Lambda_{q}}^{[0,\log q]}$ don't exhibit escape of mass. Extending later to time $2\log q$ will be an easier task.

\begin{prop}\label{prop_non_escape}
    Let $h\in [0,1]$ and $\Lambda_{q}\subseteq (\Z/q\Z)^{\times}$ such that 
    \[\liminf_{q\to\infty}\frac{\log|\Lambda_{q}|}{\log q}\geq h.\]
    Then for all $\epsilon>0$ there is a compact set $X_{\epsilon}\subset X$ so that
    \[\liminf_{q\to\infty} \delta_{\Lambda_{q}}^{[0,\log q]}(X_{\epsilon})\geq 2h-1-\epsilon.\]
\end{prop}

\begin{proof}
Let us start,  for simplicity of notation, with a general finite subset $\Theta \subseteq U$ and time $T>0$, and at the end specialize to $\Theta_{q} = \{u_{p/q}:\ p\in\Lambda_q\}$ and $T= \log(q)$.

We begin with a general upper bound for the `escaping mass' 
$$\frac{1}{|\Theta|}\sum_{u\in\Theta}\delta_{ux_0}^{[0,T]}(X^{>M}),$$ 
by separating $\Theta$ into `good' and `bad' points according to how much time the corresponding diagonal orbit spends in the cusp. 
For any $\kappa>0$, let
\[\begin{split}
\Theta^{<\kappa,M}&\coloneqq\left\{u\in \Theta : \delta_{ux_0}^{[0,T]}(X^{>M}) < \kappa \right\}=\Theta\smallsetminus C_{x_0}(X^{\leq M},T,\kappa),\\
\Theta^{\geq\kappa,M}&\coloneqq \left\{u\in \Theta : \delta_{ux_0}^{[0,T]}(X^{>M}) \geq \kappa \right\}=\Theta\cap C_{x_0}(X^{\leq M},T,\kappa).
\end{split}\]
Then it follows that:
\begin{align*}\label{Eq_est}
&\frac{1}{ |\Theta|}\sum_{u\in \Theta} \del_{ux_0}^{[0,T]}(X^{> M})\\\numberthis
=&  \frac{1}{ |\Theta|}\sum_{u\in \Theta^{<\kappa,M}} \del_{ux_0}^{[0,T]}(X^{> M})
+\frac{1}{ |\Theta|}\sum_{u\in \Theta^{\geq \kappa,M}} \del_{ux_0}^{[0,T]}(X^{> M})\\
\leq& \kappa + \frac{|\Theta^{\geq\kappa,M}|}{|\Theta|}.
\end{align*}

Hence, we want to give a lower bound to $|\Theta|$, which will be provided by the assumption of this proposition, and an upper bound to $|\Theta^{\geq \kappa,M}|$, which is given by Lemma~\ref{lem_small_step} as follows.
Given $t>t_0$ (which will be fixed later) we fix $M\geq M_{0}(t)$ as in Lemma~\ref{lem_small_step}. 
Note that by definition, 
$$\Theta^{\geq\kappa,M}\subseteq C_{x}(X^{\leq M},T,\kappa).$$ 
Therefore, we can use Lemma~\ref{lem_small_step} to cover $\Theta^{\geq \kappa,M}$  by 
\[\exp\big((1-\frac{\kappa}{2}+\frac{3\log t}{t})T\big)\]
many balls in $U$ of radius $e^{-(T-2t)}$, for $T$ large enough.

\medskip

Let us now apply for $\Theta_{q} = \{u_{p/q}:\ p\in\Lambda_q\}$ and $T=\log q$, where $\Lambda_{q}\subseteq (\Z/q\Z)^{\times}$.
Note that each ball in our covering of $\Theta_{q}^{\geq\kappa,M}$ is of radius $e^{-(T-2t)}=\frac{e^{2t}}{q}$, hence can contain at most
$\lfloor 2e^{2t}\rfloor+1$ points from $\Theta_q$. Letting $\epsilon>0$, by the proposition's assumption, for all $q$ large enough we have that $|\Theta_{q}|=|\Lambda_q|\geq q^{h-\epsilon/4}$.
Therefore, if $q$ is large enough, we have
$\frac{|\Theta_{q}^{\geq\kappa,M}|}{|\Theta_{q}|}\leq q^{f(t,q,\kappa)}$
for
\[f(t,q,\kappa)=\left(1-\frac{\kappa}{2}+\frac{3\log t}{t}\right)+\frac{\log(\lfloor 2e^{2t}\rfloor+1)}{\log q}-(h-\frac{\epsilon}{4}).\]

\medskip 

Let us now prove the proposition. Since it is trivial for $h\leq \frac{1}{2}$, we may assume that $h\in(\frac{1}{2},1]$. We can then set  $\kappa_{\epsilon}=2(1-h)+\epsilon$, and note that $\kappa_{\epsilon}\in(0,1)$ if $\epsilon$ is small enough.
Given such $\epsilon$, fix $t>t_0$ large enough so that $\frac{3\log t}{t}<\frac{\epsilon}{4}$ in addition. Then
\[\lim_{q\to\infty} f(t,q,\kappa_{\epsilon})<1-h+\frac{\epsilon}{2}-\frac{\kappa_{\epsilon}}{2}=0.\]
All together, we obtain that 
\[\frac{|\Theta_{q}^{\geq\kappa_{\eps},M}|}{|\Theta_{q}|}\leq q^{f(t,q,\kappa_\epsilon)} \overset{q\to\infty}{\longrightarrow} 0.\]

Finally, note that~\eqref{Eq_est} gives in fact a lower bound for $\delta_{\Lam_{q}}^{[0,\log q ]}(X^{>M})$. Then, we deduce
\[
\liminf_{q\to\infty}\delta_{\Lambda_{q}}^{[0,\log q]}(X^{\leq M})\geq 1-\kappa_{\eps}=2h-1-\epsilon,
\] 
concluding the proof.
\end{proof}

We now prove Theorem~\ref{Thm_nonescape}.
To do so, we need to move to averaging over $[0,2\log q]$ from the $[0, \log q]$ appearing in Proposition~\ref{prop_non_escape}. 
This is done using the following symmetry that these measures enjoy.
\begin{lem}\label{lem: half symmetry}\cite[equation (2.2) page 154]{DS}
    \label{lem:step1}
    Denote by $\tau:X\to X$ the map that sends a lattice to its dual. For $p\in(\bZ/q\bZ)^\times$, let $p'\in (\bZ/q\bZ)^{\times}$ denote the unique integer satisfying  $pp'=-1$ modulo $q$. Then for any $q\in\bN$ and any $p\in (\bZ/q\bZ)^\times$,
    \begin{equation*}
        \cavg{p/q} = \frac{1}{2}\del_{p/q}^{[0,\log q]} +
        \frac{1}{2}\tau_{\ast}\del_{p'/q}^{[0,\log q]}.
    \end{equation*}    
\end{lem}

\begin{proof}[Proof of Theorem~\ref{Thm_nonescape}]
    Let $\Lam_q\subseteq (\bZ/q\bZ)^\times$ be as in the statement of the theorem. Then by Proposition~\ref{prop_non_escape}, for any $\epsilon>0$ there is a compact set $X_{\epsilon}$ so that 
    \[\liminf_{q\to\infty} \delta_{\Lambda_{q}}^{[0,\log q]}(X_{\epsilon})\geq 2h-1-\epsilon.\]
    
    Set $\Lam_q' = \{p' :\ p\in \Lam_q\}$. Then $|\Lam_q'|=|\Lam_q|$, hence we similarly have a compact set $X_{\epsilon}^{\prime}$ so that
    \[\liminf_{q\to\infty} \tau_{\ast}\delta_{\Lambda_{q}'}^{[0,\log q]}(\tau(X_{\epsilon}'))=\liminf_{q\to\infty} \delta_{\Lambda_{q}'}^{[0,\log q]}(X_{\epsilon}')\geq 2h-1-\epsilon.\]

    By Lemma~\ref{lem:step1} we have that 
    $$\del_{\Lam_q}^{[0,2\log q]} = \frac{1}{2}\del_{\Lam_q}^{[0,\log q]} + \frac{1}{2}\tau_{\ast}\del_{\Lam_q'}^{[0,\log q]}.$$
Let $Y_{\epsilon}=X_{\epsilon}\cup\tau(X_{\epsilon}^{\prime})$. 
Putting it all together, we have
\[\liminf_{q\to\infty}\delta_{\Lambda_{q}}^{[0,2\log q]}(Y_{\epsilon})\geq 2h-1-\epsilon.\]
Note that $Y_{\epsilon}$ is compact, since $X_{\epsilon}$ is compact and $\tau$ is continuous, hence it follows that any weak* limit $\mu$ of $\delta_{\Lambda_{q}}^{[0,2\log q]}$ satisfies
\[\mu(X)\geq \mu(Y_{\epsilon})\geq 2h-1-\epsilon.\]
As $\epsilon>0$ is arbitrary, the theorem follows.
\end{proof}

\section{Continued fraction expansion of rationals}\label{sec: proof of theorem 3}
In this section we deduce Theorems~\ref{thm 1 exp}-\ref{thm 2}  from Corollary~\ref{cor: main}. The first result we prove is Theorem~\ref{thm 3}, and the rest would follow with little work.  We use the
tight relationship between the $a_t$-action on $X$ and continued fraction expansions, and a bit of the ergodicity of $\mu_X$ with respect to $a_1$. We 
will use the analysis carried out in \cite[\S 4]{DS} which in turn relies on the presentation of the aforementioned relation as in the book
\cite{EW}. 

With a slight abuse of notation, for $q\in \bN$ and $p\in (\bZ/q\bZ)^{\times}$ we will let $\delta_{p/q}\defn \delta_{u_{p/q}x_0}$, where the latter is the probability measure supported on $\{u_{p/q}x_0\}$ as before.
 We will be using the following result from \cite[\S4]{DS}:
\begin{lem}[Lemma 4.11 and Theorem 4.12 from \cite{DS}]
    \label{lem: DS}
    Let $Q\subseteq \bN$ be a sequence of denominators and assume that for 
    $q\in Q$ we are given $p_q\in(\bZ/q\bZ)^\times$ satisfying 
    $\cavg{p_q/q}\to \mu_X$. Then $\frac{\on{len}(p_q/q)}{\log q} \to \frac{12 \log 2}{\pi^2}$ and for any finite string of integers $\bw=(w_1,\dots, w_k)\in \bN^k$, 
    $D_\bw(p_q/q)\to D_\bw$.
\end{lem}
Here $\on{len}(p/q)$ and $D_\bw$ are defined in \eqref{eq: length of cfe} and \eqref{eq: asymptotic density of w}.
\begin{remark}\label{remark:gauss_measure}
    We note that in \cite[Lemma 4.11]{DS} there is no mention of the asymptotic densities $D_\bw, D_\bw(p/q)$. Still, the claim as we stated it follows easily; Let $I_\bw$ denote the subinterval of the unit interval consisting 
    of those numbers whose $\cfe$ starts with $\bw$. Then, in the notations of~\cite[Lemma 4.11]{DS}, it follows from the ergodicity of $\nu_{\on{Gauss}}$ with respect to the Gauss map  that 
    $D_\bw= \nu_{\on{Gauss}}(I_\bw)$ (c.f.\ the introduction of~\cite{AS}), and it follows 
    straight from the definitions that  
    $D_\bw(p_q/q) = \frac{\on{len}(p_q/q)}{\on{len}(p_q/q) - k+1}\nu_{p_q/q}(I_\bw)$. 
    Since $\nu_{p_q/q}\to\nu_{\on{Gauss}}$, we have  
    $\nu_{p_q/q}(I_\bw)\to \nu_{\on{Gauss}}(I_\bw)$. As $\on{len}(p_q/q)\asymp \log q$ we deduce that indeed
     $$D_\bw(p_q/q)\to D_\bw$$
    as claimed.
\end{remark}

Recall that in Corollary~\ref{cor: main} we showed that given the condition $\frac{\log |\Lam_q|}{\log q}\to 1$ we obtain that $\delta_{\Lambda_q}^{[0, 2\log q ]} \wstar \mu_X$. As can be seen in Lemma ~\ref{lem: DS}, we would like to consider measures $\delta_{p_q/q}^{[0,2\log q]}$ instead. This will be done by moving from averages over $\Lambda_q$ to a claim about `almost every' $p\in \Lambda_q$.
This step is actually quite general in nature. To formulate it in a more general setting, we need the following definition.
\begin{defi}\label{defi:average}
    For a finite non-empty set $\Omega$ of probability Radon measures on $X$, we denote its average by
    $$Avg(\Omega) = \frac{1}{|\Omega|} \sum_{\nu \in \Omega} \nu.$$
    We also define $Avg(\emptyset)$ to be the zero measure.
    \end{defi}
    \begin{defi}
    If $\Omega_q$ is a sequence of sets as in Definition~\ref{defi:average}, we say that they are \textbf{uniformly almost invariant} if any possible weak* limit point of $Avg(\Omega_q')$ is $a_t$-invariant,  for any choice of non-empty subsets $\Omega_q'\subseteq \Omega_q$.
\end{defi}
\begin{remark}
    The application we have in mind for this last definition is $\Omega_{q}=\{\delta_{p/q}^{[0,2\log q]}:\ p\in \Lambda_{q}\}$, so that $Avg(\Omega_{q})=\delta_{\Lambda_{q}}^{[0,2\log q]}$.
    The claim that $\Omega_{q}$ is `uniformly almost invariant' follows from the fact that each measure $Avg(\Omega_{q}')$ is a  diagonal continuous average of a probability measure, over a time interval of length $2 \log q$ which approaches infinity.
\end{remark}

We prove the following proposition, which is an application of the ergodicity of the Haar measure on $X$.

\begin{prop}
    \label{lem:step2}
    Let $\Omega_q$ be a sequence of uniform almost invariant sets such that $Avg(\Omega_q)\wstar\mu_X$. Then there exist $\Omega_q'\subseteq \Omega_q$ such that $\frac{|\Omega_q'|}{|\Omega_q|}\to 1$ and for any choice of a sequence $\nu_q \in \Omega_q'$ we have that $\nu_q \wstar\mu_X$.
\end{prop}
\begin{proof}
    We first claim that the conclusion of the proposition holds if for any $\eps>0$ and $f\in C_c(X)$, there exists an integer $q(\eps,f)\in\bN$ such that for all $q\geq q(\eps,f)$, 
    $$\frac{1}{\av{\Omega_q}}\#\set{\nu \in \Omega_q: \av{\nu(f)-\mu_X(f)}>\eps}\leq 2\eps.$$ Indeed, since $C_c(X)$ is a separable Banach space, there exists a countable dense subset $\{f_i\}_{i\in\bN}$ of $C_c(X)$. Set $q_0=1$ and for each $k\in\bN$ define $$q_k:=\max\set{q_{k-1}+1,q(1/k^2,f_i):1\leq i\leq k }.$$ It follows from the claim that for any $q\geq q_k$ and $i=1,\dots,k$,
    \eqlabel{Eq_countbound}{
    \frac{1}{\av{\Omega_q}}\#\set{\nu \in \Omega_q: \av{\nu(f_i)-\mu_X(f_i)}>\frac{1}{k^2}}\leq \frac{2}{k^2}.
    } For any $q_k \leq q< q_{k+1}$, we set $$\Om_q' = \set{\nu\in \Om_q: \av{ \nu(f_i)-\mu_X(f_i)}\leq \frac{1}{k^2}\ \forall i=1,\dots,k}.$$ By \eqref{Eq_countbound}, we have $\av{\Om_q'}\geq \left(1-\frac{2}{k}\right)\av{\Om_q}$, hence $\frac{\av{\Om_q'}}{\av{\Om_q}}\to 1$. 
    Note that  for any choice of a sequence $\nu_q \in \Omega_q'$, we have $\nu_{q}(f_i)\overset{q\to\infty}{\to}\mu_{X}(f_i)$ for all $i\in\bN$.
    Since $\{f_i\}_{i\in\bN}$ is a dense subset of $C_c(X)$, it follows that $\nu_q \wstar\mu_X$. This proves the reduction.\\

    Let us now prove the claim. Assume, for the sake of contradiction, that the claim does not hold. Then 
there exist $\eps>0$, a function $f\in C_c(X)$, and an unbounded subset $Q\subseteq \bN$ of $q$'s such that for all $q\in Q$ 
    \begin{equation}\label{eq: using ergodicity}
    \frac{1}{\av{\Omega_q}}\#\set{\nu \in \Omega_q: \av{
    \nu(f)-\mu_X(f)}>\eps}> 2\eps.
    \end{equation}
    It follows from \eqref{eq: using ergodicity} that there is an infinite subsequence 
    $Q'\subseteq Q$ such that along $Q'$, at least one of the following statements 
    hold:
    \begin{enumerate} 
    \mathitem \label{eq: say positive}
      \[\frac{1}{\av{\Omega_q}} \#\set{\nu\in \Omega_q: 
    \nu(f)>\mu_X(f) +\eps}> \eps.\] 
    \mathitem 
    \[\frac{1}{\av{\Omega_q}}\#\set{\nu\in \Omega_q: 
    \nu(f)<\mu_X(f) -\eps}> \eps.\]
    \end{enumerate}

    Let us assume without loss of generality that option \eqref{eq: say positive} holds.  For $q\in Q'$, let \[A_q = \set{\nu\in \Omega_q: \nu(f) >\mu_X(f)+\eps},\qquad  B_q = \Omega_q\sm A_q.\]
    Since $\frac{|A_q|}{|\Omega_q|}\in(\epsilon,1]$ and $\frac{|B_q|}{|\Omega_q|}=1-\frac{|A_q|}{|\Omega_q|}$ are bounded, we may take an infinite subsequence $Q''\subseteq Q'$ so that $\frac{|A_q|}{|\Omega_q|}$ and $\frac{|B_q|}{|\Omega_q|}$ converge to some numbers $\alpha, (1-\alpha)$, respectively. 
    Note that $\alpha\geq \epsilon>0$. Next, recall that the set of all Radon measures with mass at most $1$,  is sequentially compact in the weak* topology.
    Therefore, we may assume that $Q''$ was chosen so that $Avg(A_q)\wstar \nu_1$ and $Avg(B_q)\wstar \nu_2$ as $q\to \infty$ in $Q''$, for some measures $\nu_1,\nu_2$.
    As our sets $\Omega_q$ are uniformly almost invariant, these limits $\nu_1,\nu_2$ are $a_1$-invariant measures.
     
    Let us now present $Avg(\Omega_q)$ as a convex
    combination as follows:
    \begin{equation}
        \label{eq:convex combination}
        Avg(\Omega_q) =
        \frac{\av{A_q}}{\av{\Omega_q}}Avg(A_q)+\frac{\av{B_q}}{\av{\Omega_q}}Avg(B_q).
    \end{equation}
    Taking the limit over the sequence $Q''$ 
    we conclude from \eqref{eq:convex combination}, using the assumed convergence $Avg(\Omega_q)\wstar\mu_X$,
    that 
    $$\mu_X = \al\nu_1 + (1-\al)\nu_2.$$
    If $\alpha=1$, then $\mu_X=\nu_1$. Otherwise both $\alpha,1-\alpha>0$, so the ergodicity of $\mu_X$ with respect to the $a_1$-action also implies that 
    $\mu_X=\nu_1$. 
    However, in either case, by definition of $A_{q}$ we have $Avg(A_q)(f)>\mu_X(f)+\epsilon$ and hence \[\mu_X(f)=\nu_1(f)=\lim_{q\in Q''}Avg(A_q)(f)\geq \mu_X(f)+\epsilon,\]
    in contradiction.  This finishes the proof of the claim and with it the proposition.
\end{proof}

We can now turn to prove the main results of this paper concerning the continued fraction expansions of rationals. We start from Theorem~\ref{thm 3}, from which Theorem~\ref{thm 1 exp} would then easily follow. 

\begin{proof}[Proof of  Theorem~\ref{thm 3}] 
Let $\Lam_q\subseteq (\bZ/q\bZ)^\times$ be as in the statement of the theorem. By Corollary~\ref{cor: main},  $\del_{\Lam_q}^{[0, 2\log q]} \wstar \mu_X$. Hence we can apply Proposition~\ref{lem:step2} and find a set 
$\hat{\Lam}_q\subseteq \Lam_q$ of size $\frac{|\hat{\Lam}_q|}{|\Lam_q|}\to 1$, so that for any choices of $p_q\in \hat{\Lam}_q$ we have that 
$\cavg{p_q/q}\to \mu_X$.
Consequently, by Lemma~\ref{lem: DS}, for every such sequence $p_q$ we have 
that 
\begin{equation}\label{eq: lenngths converge}
\frac{\on{len}(p_q/q)}{\log q}\to \frac{12\log 2}{\pi^2},
\end{equation}
and for any finite string of integers $\bw\in \bN^k$, 
\begin{equation}\label{eq: desities converge}
D_\bw(p_q/q)\to D_\bw. 
\end{equation}

Assume by way of contradiction that for some $\eps>0$ we have
$$\bP|_{\Lam_q}(\av{\frac{\on{len}(p/q)}{\log q} - \frac{12\log 2}{\pi^2}}>\eps)\nrightarrow 0,$$
where $p\in \Lam_q$ is chosen uniformly at 
random.
This means that along a subsequence $Q\subseteq \bN$ we have 
that the set 
$$A_q = \set{p\in \Lam_q:\av{\frac{\on{len}(p/q)}{\log q} - \frac{12\log 2}{\pi^2}}>\eps}$$ occupies a positive proportion of $\Lam_q$ (that is
$\av{A_q}/\av{\Lam_q}>\eta$ for some positive $\eta$), and in particular, $A_q\cap \hat{\Lam}_q\ne \varnothing$ for all large $q$. We choose $p_q\in \hat{\Lam}_q\cap A_q$
and arrive to a contradiction arising from \eqref{eq: lenngths converge} and the definition of $A_q$. 

The argument for the asymptotic frequency is identical, hence is omitted.  
\end{proof}

Next, we deduce Theorem~\ref{thm 1 exp} from Theorem~\ref{thm 3}. We leave it to the reader to prove that the inverse deduction holds as well, so the two results are in fact equivalent.
\begin{proof}[Proof of Theorem~\ref{thm 1 exp}]
We only prove item~\ref{item:1.2.2} concerning the length of the c.f.e. The argument for the asymptotic frequency is the same, hence is omitted.
For any  $\eps>0$ and $q\in\bN$, let
\[\text{Bad}_{\eps}(q)=\left\{j\in (\bZ/q\bZ)^{\times}:\ \av{\frac{\on{len}(j/q)}{\log q} - \frac{12\log 2}{\pi^2}} >\eps\right\}.\]
Note that an equivalent formulation of the theorem is the following inequality, which we now prove: 
\eqlabel{eq:bad_eps}{\limsup_{q\to \infty}\frac{\log\big(|\text{Bad}_\epsilon (q)| / \varphi(q)\big)}{\log(q)}<0.}

Since \[\frac{\log(\varphi(q))}{\log(q)} \overset{q\to\infty}{\to} 1,\]
\eqref{eq:bad_eps} is equivalent to showing that
\[\limsup_{q\to \infty}\frac{\log(|\text{Bad}_\epsilon (q)|)}{\log(q)}<1.\]
If this inequality would not hold, we could find a strictly increasing sequence $(q_n)_{n\in\bN}$ such that 
\[\frac{\log(|\text{Bad}_{\eps}(q_n)|)}{\log q_n}{\to} 1.\]
Define $\Lam_{q}\subseteq (\Z/q\Z)^{\times}$ for all $q\in \bN$ by
\[\Lam_{q}=\begin{cases}
    \text{Bad}_{\eps}(q) & q=q_n\text{ for some $n$}\\
    (\Z/q\Z)^{\times} & \text{otherwise}.
\end{cases}\]
Then we can apply Theorem~\ref{thm 3} and deduce that
\[\bP|_{\Lam_{q}}\pa{\Lam_{q}\cap \text{Bad}_{\eps}(q)}\to 0 \textrm{ as }q\to \infty.\]
In particular,
 \eqlabel{eq probability of bad}{\bP|_{\text{Bad}_{\eps}(q_n)}\pa{\text{Bad}_{\eps}(q_n)}\to 0 \textrm{ as }n\to \infty.}
 However, the LHS in~\eqref{eq probability of bad} is identically $1$, in contradiction. \end{proof}

We end up this paper with deducing Theorem~\ref{thm 2}, namely our application for rational numbers with prime numerators.
\begin{proof}[Proof of Theorem~\ref{thm 2}]
    For any $q$, let 
    $$\Lam_q: = \set{ p\in(\bZ/q\bZ)^{\times}: p\textit{ is prime}}.$$
    We will show that 
    \begin{equation}\label{eq: 1711}
    \lim_{q\to \infty} \frac{\log\av{\Lam_q}}{\log q} = 1.
    \end{equation}
    This would allow us to apply Theorem~\ref{thm 3} with this choice of $\Lam_q$, and the resulting statement is precisely Theorem~\ref{thm 2}. 

    Equation~\eqref{eq: 1711} is just a simple consequence of the prime number theorem. We include the argument for completeness. 
    Let $\pi(q)$ denote the number of primes $\le q$. Then
    \[\pi(q) = \av{\Lam_q}+\#\set{p:p \textit{ is prime, } p|q}.\]
    Since the number of prime divisors of $q$ is bounded by $\log q$ we get that
    $ \pi(q) \le \av{\Lam_q}+\log q$. 
    In particular,
    $$\log \av{\Lam_q}\ge \log(\pi(q) - \log q) = \log (\pi(q)(1 - \frac{\log q}{\pi(q)})) \sim \log \pi(q),$$
    where the asymptotic relation follows from the fact that $\frac{\log q}{\pi(q)}\to 0$ by the prime number theorem. So, in order to show~\eqref{eq: 1711}, it is suffices to show that $\frac{\log \pi(q)}{\log q}\to 1$. Indeed, again 
    by the prime number theorem, we know that
    $\displaystyle{\lim_{q\to \infty}} \frac{\pi(q)\log q}{q} = 1.$ Taking logarithms we obtain that 
    $$ \lim_{q\to\infty}\Big( \log \pi(q)  - \log q + \log \log q \Big) = 0.$$
    Dividing by $\log q$ we see that $\frac{\log \pi(q)}{\log q}\to 1$ as needed.
\end{proof}

\def\cprime{$'$} \def\cprime{$'$} \def\cprime{$'$}
\providecommand{\bysame}{\leavevmode\hbox to3em{\hrulefill}\thinspace}
\providecommand{\MR}{\relax\ifhmode\unskip\space\fi MR }
\providecommand{\MRhref}[2]{%
  \href{http://www.ams.org/mathscinet-getitem?mr=#1}{#2}
}
\providecommand{\href}[2]{#2}

\end{document}